\documentclass[11pt,twoside,reqno]{amsart}
\usepackage{amsmath}
\usepackage{amsthm}
\usepackage{amsfonts,amssymb}
\usepackage{bbm}
\usepackage{latexsym}
\usepackage{mathrsfs}
\usepackage[all]{xy}
\usepackage{url}
\usepackage[pdftex]{graphicx}
\usepackage{float}
\usepackage{hyperref}
\usepackage{amssymb,yfonts}
\usepackage{fontenc}

\usepackage{mathdots}

\usepackage[boxsize=23pt]{ytableau}

\usepackage[utf8]{inputenc}
\usepackage[english]{babel}
\usepackage{lipsum}
\usepackage{amsthm}
\usepackage{todonotes}
\usepackage[vcentermath]{youngtab}

\theoremstyle{plain}
\newtheorem{lema}{Lemma}[section]
\newtheorem{prop}[lema]{Proposition}
\newtheorem{teo}[lema]{Theorem}

\newtheorem{coro}[lema]{Corollary}
\theoremstyle{remark}

\newtheorem{remark}[lema]{Remark}

\theoremstyle{definition}

\newtheorem{ej}[lema]{Example}

\def\a\Si{{\rm{a}\Sigma }}
\def\w\Si{{\rm{w}\Sigma }}

\def\w {{\textrm {w}}}

\def\min{{\text{min}}}
\def\max{{\text{max}}}
\def\des{{\mathrm{des}}}
\def\Des{{\mathrm{Des}}}

\def\a{\mathrm{a}}
\def\Int{{\mathrm{Int}}}

\def\rank{\mathrm{rank}}

\def\F{{\mathcal{F}}}
\def\R{{\mathcal{R}}}

 \newcommand{\De}{\Delta}
 \newcommand{\Si}{\Sigma}
 \newcommand{\si}{\sigma}

\begin{document}

\title[On Interval Subdivision]{The $f$- and $h$-vectors of Interval Subdivisions}

\author[I. Anwar]{Imran Anwar}
\author[S. Nazir]{Shaheen Nazir}

\address{Abdus Salam School of Mathematical Sciences\\
Government College  University\\
Lahore and National Center for Mathematics, Pakistan}

\address{Department of Mathematics\\
Lahore University of Management Sciences\\
Lahore, Pakistan}

\email{imrananwar@sms.edu.pk}

\email{shaheen.nazir@lums.edu.pk}

\begin{abstract}
The interval subdivision $\Int(\Delta)$ of a simplicial complex
$\Delta$ was introduced by Walker. We give a complete
combinatorial description of the entries of the transformation
matrices from the $f$- and $h$-vectors of $\Delta$ to the $f$- and
$h$-vectors of $\Int(\Delta)$.  We show that if $\Delta$ has a
non-negative $h$-vector then the $h$-polynomial of its interval
subdivision has only real roots.  As a consequence, we prove the
Charney-Davis conjecture for $\Int(\Delta)$, if $\Delta$ has a
non-negative reciprocal $h$-vector.

\end{abstract}\subjclass[2010]{}


\maketitle

\section{Introduction}
In this paper, we study the behavior of  the enumerative invariants of a simplicial complex under the interval subdivision,
 introduced by Walker \cite{walker1988canonical}.  This work is motivated from the work of Brenti and Welker
about the barycenteric subdivision of simplicial complexes \cite{brenti2008f}.
 The enumeration data e.g., $f$-, $h$-, $\gamma$-, $g$-vectors of barycentric subdivision of a simplicial complex has been extensively
studied in the literature, see \cite{stanley1992subdivisions, brenti2008f, kubitzke2009lefschetz,  murai2010face, nevo2011gamma, petersen2015eulerian}.
  Let $\Delta$ be a $(d-1)$-dimensional simplicial complex  on the
ground set $V$. The interval subdivision $\Int(\Delta)$ of
$\Delta$ is the simplicial complex on the ground set
$I(\Delta\setminus \{\emptyset\})$, where $I(\Delta\setminus
\{\emptyset\}):=\{[A, B]\ | \ \emptyset\neq A \subseteq B \in
\Delta \}$
 is  partially  ordered by
inclusion, that is $[A,B]\subseteq [A',B'] \in I(\Delta\setminus
{\emptyset}) $ if and only if $A'\subseteq A\subseteq B\subseteq B'$.
By Walker \cite[Theorem 6.1(a)]{walker1988canonical}, the simplicial
complex of all chains in the partially ordered set
$I(\Delta\setminus \{\emptyset\})$ is a subdivision of $\Delta$. It can
be noted that this subdivision is the special case $N=1$ of the
simplicial complex considered in \cite[Fig.
1.2]{cheeger1984curvature}. The aim of this article is to analyze the behavior of the $f$- and $h$-vectors moving from $\Delta$ to
$\Int(\Delta)$. In the  main result of this paper, we show that if the $h$-vector of a simplicial complex
is non-negative then the $h$-polynomial of its interval subdivision has only real roots. This is done by showing that a certain class of  the refined $j$-Eulerian  polynomials of type $B$  (defined in Section 2) is compatible.
 \\
The paper is organized as follows. In Section $2$, we give a formula of the $f$-vector of the interval subdivision $\Int(\Delta)$ in terms
 of the $f$-vector of the simplicial complex $\Delta$.   In Section $3$, we study the transformation matrix of the  $h$-vector of the interval subdivision $\Int(\Delta)$.
 We give an interpretation of the coefficients of the  transformation matrix of the $h$-vector in terms of the refined  Eulerian numbers of type $B$.
 It is well known that these coefficients are $j$-Eulerian numbers of type $A$ in the case of barycentric subdivision, see  \cite{stanley1986enumerative, brenti2008f}. Along the way, we also show that if the $h$-vector of the simplicial complex $\Delta$ satisfies the Dehn-Sommerville relations, so does the $h$-vector of  interval subdivision $\Int(\Delta)$.
 Moreover, we show some simple results on the properties of  the $f$- and $h$-vector transformation matrices. In the sequel, we show that these transformation matrices are diagonalizable and similar.  In \cite{brenti2008f}, the authors studied  the limiting behavior of roots of the $h$-polynomials of successive barycentric subdivision of $\De$ and proved that these roots are depending only on the dimension of $\De$. Later, in \cite{delucchi2012face}, the authors improved and generalized this result. We also discuss the limiting behavior of roots of the $f$-polynomials of successive interval subdivision of $\De$ and proved the similar result for successive interval subdivisions.
 Section $4$ is devoted to the proof of  our main Theorem \ref{h-poly}. It states that if $\Delta$ has a non-negative
$h$-vector then the $h$-polynomial of its interval subdivision has
only real zeros. Additionally, we  prove that the refined $j$-Eulerian
polynomials of type $B$ are real-rooted. As a
consequence of the Theorem \ref{h-poly}, we derive the
Charney-Davis conjecture for the interval subdivision of a  simplicial complex
$\Delta$ with a non-negative reciprocal $h$-vector in Corollary \ref{Chraney}.

\section{The $f$-vector Transformation}
Throughout this article, $\Delta$ represents a $(d-1)$-dimensional
simplicial complex on the ground set $V=[n]$. In this section, we
will describe the transformation  sending the $f$-vector of
$\Delta$ to the  $f$-vector of  $\Int(\Delta)$. Recall that
the vector $f(\Delta):=(f_{-1}(\Delta),f_0({\Delta}),\ldots,f_{d-1}({\Delta}))$, where $f_i(\Delta)$ is the number of
$i$-dimensional faces of
$\Delta$ is  called the \textit{$f$-vector} of $\Delta$ with $f_{-1}(\Delta)=1$ (for $\dim \emptyset=-1$).  \\
By the definition of $\Int(\Delta)$, an $l$-dimensional face in
$\Int(\Delta)$ is  a chain
$$[A_0,B_0]\subset [A_1,B_1]\subset \cdots \subset [A_l,B_l]$$ of
intervals of length $l$ in $I(\Delta\setminus \{\emptyset\})$. As a warm-up we start with a description of $f_0(\Int(\Delta))$.
\subsection*{$f_0({\Int(\Delta)})$:}
$f_0({\Int(\De)})$ is the  number of  intervals
$[A,B]$, where $A\subseteq B$ for all $A,B\in \Delta\setminus
\{\emptyset\}$.  For any $B\in \Delta\setminus \{\emptyset\}$, all distinct
 subsets of $B$ (excluding $\emptyset$) give rise to
distinct intervals terminating at $B$.  For  fixed $B\in \Delta$, the number of intervals of the form $[A,B]$  is $2^{|B|}-1$.  Since there are $f_{l-1}(\Delta)$ choices for $B$ with $|B|=l$, the number
of all possible  intervals in $I(\Delta\setminus \emptyset)$ will be

$$(2^0-1)f_{-1}(\Delta)+(2^1-1)f_0(\Delta) + (2^2-1)f_1({\Delta}) + (2^3-1)f_2({\Delta})+ \cdots +(2^{d}-1)f_{d-1}({\Delta}) .$$
Thus, \begin{equation}\label{0}f_0({\Int(\Delta)})=
\sum_{k=0}^d (2^{k}-1)f_{k-1}({\Delta}).\end{equation}
Now we turn to the description of $f_k({\Int(\Delta)})$ in general.
\subsection*{$f_k({\Int(\Delta)})$:} To compute $f_k({{\Int(\Delta)}})$, for $k\ge0$, let us introduce
some notations. It is easily seen that the number of chains of length $k$ terminating at
$[A,B]$ only depends on $\alpha = |B\setminus A|$.
 Let $Q_k^{\alpha}$ denote the number of  chains of
intervals of length $k$ terminating at some fixed interval $[A , B]$,
where $\alpha =|B\setminus A|$. By definition, $Q_k^{\alpha}=0$ for $\alpha < k$ and $Q_0^{\alpha}=1$ for all $\alpha$.\\
We group the $k$-chains in $ I(\Delta\setminus \emptyset)$ according to the top element $[A,B]$
of the chain. For a fixed $[A,B]$, we have $Q_k^{l-t}$ chains of length $k$ terminating in
$[A,B]$, where $t = |A|$ and $l =  |B|$. There are $f_{l-1}(\Delta)$ choices for $B$ with
$|B| = l$ and for a fixed $B$ we have ${l\choose t}$ subsets $A\subseteq B$
with  $|A| = t$.
Hence, we have

\begin{equation}\label{1}
  f_k({\Int(\Delta)})=\sum_{l=0}^{d}\big{[}\sum_{t=1}^{l}  {{l}\choose{t}}Q_k^{l-t}\big{]}f_{l-1}({\Delta}).
  \end{equation}
In the next lemma, we give an expression for  $Q_k^{\alpha}$.
\begin{lema}
  The formula for $Q_k^{\alpha}$ is given as
\begin{equation}\label{ind}
  Q_k^{\alpha}= \sum_{i=0}^{k}
(-1)^i{{k}\choose{i}}(1+2(k-i))^{\alpha}.
\end{equation}
\end{lema}
\begin{proof}
 We  will prove \eqref{ind}  by induction on
$k$. The case  $k=0$ follows from the definition.  Now,  suppose that \eqref{ind} is true for $k-1$. To compute $Q_k^{\alpha}$, we intend to count all
$k$-chains of intervals terminating at  some fixed interval $[A,B]$ with $|B\setminus A|=\alpha$. Let $B=A\cup\{a_1,a_2,\ldots, a_{\alpha}\}$. The
intervals strictly contained in $[A,B]$ are of the form $[A\cup\{a_{i_1}, \ldots, a_{i_t}\} ,A\cup\{a_{i_1},\ldots,
a_{i_{t+s}}\}]$ unless $t=0$ and $s=\alpha$. There are ${\alpha \choose t}{\alpha-t \choose s}$ choices for intervals of the form $[A\cup\{a_{i_1}, \ldots, a_{i_t}\} ,A\cup\{a_{i_1},\ldots,
a_{i_{t+s}}\}]$ contained in $[A,B]$, and  the number of  all  chains of length $k-1$
terminating at $[A\cup\{a_{i_1},\ldots, a_{i_t}\}
, A\cup\{a_{i_1},\ldots,  a_{i_{t+s}}\}]$ is $ Q_{k-1}^{s}$.  Hence for fixed $\alpha$ and $k$, we have the following recurrence relation
$$Q_k^{\alpha}=\sum_{i=0}^{\alpha} {\alpha\choose i}\big{[}\sum_{j=0}^{\alpha-i}{\alpha-i\choose j} Q_{k-1}^j\big{]}-Q_{k-1}^{\alpha}.$$
Since \eqref{ind} is true for $k-1$, so substitute the formula of $Q^{\alpha}_{k-1}$ in the above expression
$$
  Q_k^{\alpha}=\sum_{i=0}^{\alpha} {\alpha\choose i}\big{[}\sum_{j=0}^{\alpha-i}{\alpha-i\choose j} \sum_{m=0}^{k-1}
(-1)^m{{k-1}\choose{m}}(1+2(k-1-m))^{j} \big{]}$$
$$- \sum_{m=0}^{k-1}
(-1)^m{{k-1}\choose{m}}(1+2(k-1-m))^{\alpha}.
$$

Using the  binomial formula twice  (first  summing over $j$ and then over $i$), we have
$$Q_k^{\alpha}=\sum_{m=0}^{k-1}(-1)^m {{k-1}\choose m}\big[(1+2k-2m)^{\alpha}-(1+2k-2(m+1))^{\alpha}\big]$$
Now, using the identity ${{k-1}\choose{m}}+{{k-1}\choose {m-1}}={k\choose m}$, we get
$$Q_k^{\alpha}=(1+2k)^{\alpha}+\sum_{m=1}^{k-1}(-1)^m{{k}\choose m}(1+2k-2m)^{\alpha}+(-1)^k$$
which gives the required form.
 \end{proof}

Thus,  we have the $f$-vector transformation as follows.
\begin{teo}\label{f} Let $\Delta$ be a $(d-1)$-dimensional simplicial complex. Then
  \begin{equation}\label{2}
  f_k({\Int(\Delta)})=\sum_{l=0}^{d}\sum_{i=0}^{k} (-1)^i {{k}\choose{i}}\big{[}(2+2k-2i)^{l}-(1+2k-2i)^{l}\big{]}f_{l-1}(\Delta).
  \end{equation}
  for $0\leq k\leq d-1$ and  $f_{-1}({\Int(\Delta)})=f_{-1}(\Delta)=1$.
\end{teo}
\begin{proof}  Using the expressions \eqref{1} and \eqref{ind}, we have the following
\begin{align*}
  f_k({\Int(\Delta)})=&\sum_{l=0}^{d}\big{[}\sum_{t=1}^{l}  {{l}\choose{t}}\sum_{i=0}^{k}
(-1)^i{{k}\choose{i}}(1+2(k-i))^{l-t}\big{]}f_{l-1}({\Delta})\\
  =& \sum_{l=0}^{d}\big{[}\sum_{i=0}^{k}
(-1)^i{{k}\choose{i}}\sum_{t=1}^{l}  {{l}\choose{t}}(1+2(k-i))^{l-t}\big{]}f_{l-1}({\Delta})
  \end{align*}
Summing over $t$ and using the binomial formula, we get the required expression \eqref{2}.
\end{proof}
\begin{remark}
The formula \eqref{2} can be rewritten(using the binomial formula) as
 $$f_k({\Int(\Delta)})=\sum_{l=0}^{d}\sum_{i=0}^{k} (-1)^i {{k}\choose{i}}\big{[}\sum_{j=0}^l{{l}\choose {j}}(2^l-2^j)(k-i)^j\big{]}f_{l-1}(\Delta).$$
Using the explicit formula for  Stirling numbers $S(j,k)$ of second kind  $$k!S(j,k)=\sum_{i=0}^{k}(-1)^i{{k}\choose {i}}(k-i)^j,$$
 we get

\begin{equation}\label{3}
  f_k({\Int(\Delta)})=\sum_{l=0}^{d}\sum_{j=0}^{l}  {{l}\choose{j}}k!S(j,k)\big{(}2^l-2^j\big{)}f_{l-1}(\Delta).
  \end{equation}

\end{remark}
The transformation of the $f$-vector of
$\Delta$ to the $f$-vector of interval subdivision $\Int(\Delta)$ is given by the matrix:
 $$\mathcal{F}_d=[b_{k,l}]_{0\leq k,l\leq d},$$  where
 $$b_{0,l}=\left\{
               \begin{array}{ll}
                 1, & \hbox{$l=0$;} \\
                 0, & {l>0.}
               \end{array}
             \right.
$$ and for $1\leq k\leq d$, we have
\begin{equation}\label{matrix}
  b_{k,l}=\sum_{j=0}^{k-1}(-1)^j{k-1 \choose
j}[(2k-2j)^l-(2k-2j-1)^l]
\end{equation}
Here, we include an example to demonstrate the above computed
transformation.
\begin{figure}[h]

\centering
\begin{tikzpicture}[scale=0.6]
\filldraw[color=black, fill=gray!30,ultra thick]  (-3,0) -- (13,0) -- (5,13)  -- cycle;
\filldraw[black] (-3,0) circle (8pt) node[anchor=north]{};
 \fill[black,font=\large]
                    (-3,-0.5) node [below] {$[1,1]$};
\filldraw[black] (13,0) circle (8pt) node[anchor=west]{};
\fill[black,font=\large]
                    (13,-0.5) node [below] {$[3,3]$};
\filldraw[black] (5,13) circle (8pt) node[anchor=west]{};
\fill[black,font=\large]
                    (5,13.5) node [above] {$[2,2]$};
\filldraw[black] (1,0) circle (8pt) node[anchor=west]{};
\fill[black,font=\large]
                    (1,-0.5) node [below] {$[1,13]$};
\filldraw[black] (5,0) circle (8pt) node[anchor=west]{};
\fill[black,font=\large]
                    (5,-0.5) node [below] {$[13,13]$};
\filldraw[black] (9,0) circle (8pt) node[anchor=west]{};
\fill[black,font=\large]
                    (9,-0.5) node [below] {$[3,13]$};
\filldraw[black] (-1,3.3) circle (8pt) node[anchor=west]{};
\fill[black,font=\large]
                    (-1.3,3.3) node [left] {$[1,12]$};
\filldraw[black] (1,6.6) circle (8pt) node[anchor=west]{};
\fill[black,font=\large]
                    (0.7,6.6) node [left] {$[12,12]$};
\filldraw[black] (3,9.9) circle (8pt) node[anchor=west]{};
\fill[black,font=\large]
                    (2.7,9.9) node [left] {$[2,12]$};
\filldraw[black] (11,3.3) circle (8pt) node[anchor=west]{};
\fill[black]
                    (11.3,3.3) node [right] {$[3,23]$};
\filldraw[black] (9,6.6) circle (8pt) node[anchor=west]{};
\fill[black,font=\large]
                    (9.3,6.6) node [right] {$[23,23]$};
\filldraw[black] (7,9.9) circle (8pt) node[anchor=west]{};
\fill[black,font=\large]
                    (7.3,9.9) node [right] {$[2,23]$};
\filldraw[black] (5,4.5) circle (8pt) node[anchor=west]{};
\fill[black,font=\large]
                    (5,4.8) node [above] {$[123,123]$};
\filldraw[black] (1,2.25) circle (8pt) node[anchor=west]{};
\fill[black,font=\large]
                    (1.1,2.55) node [above] {$[1,123]$};
\filldraw[black] (9,2.25) circle (8pt) node[anchor=west]{};
\fill[black,font=\large]
                    (9,2.55) node [above] {$[3,123]$};
\filldraw[black] (5,2.25) circle (8pt) node[anchor=west]{};
\fill[black,font=\large]
                    (5,2.55) node [above] {$[13,123]$};
\filldraw[black] (5,8.75) circle (8pt) node[anchor=west]{};

\filldraw[black] (3,5.5) circle (8pt) node[anchor=west]{};
\fill[black,font=\large]
                    (3,5.8) node [above] {$[12,123]$};
\filldraw[black] (7,5.5) circle (8pt) node[anchor=west]{};
\fill[black,font=\large]
                    (7,5.8) node [above] {$[13,123]$};
\draw[black, ultra thick] (-3,0) -- (9,6.6);
\draw[black, ultra thick] (5,13) -- (5,0);
\draw[black, ultra thick] (13,0) -- (1,6.6);
\draw[black,ultra thick] (1,6.6) -- (1,0);
\draw[black, ultra thick] (9,6.6) -- (9,0);
\draw[black,ultra thick] (-1,3.3) -- (5,0);
\draw[black,ultra thick] (11,3.3) -- (5,0);
\draw[black, ultra thick] (1,2.25) -- (9,2.25);
\draw[black,  ultra thick] (1,2.25) -- (5,8.75);
\draw[black, ultra thick] (9,2.25) -- (5,8.75);
\draw[black, ultra thick] (3,9.9) -- (5,8.75);
\draw[black, ultra thick] (7,9.9) -- (5,8.75);
\draw[black,ultra thick] (1,6.6) -- (5,8.75);
\draw[black, ultra thick] (9,6.6) -- (5,8.75);

\end{tikzpicture}

\caption{Interval Subdivision of the $2$-Simplex}
\end{figure}
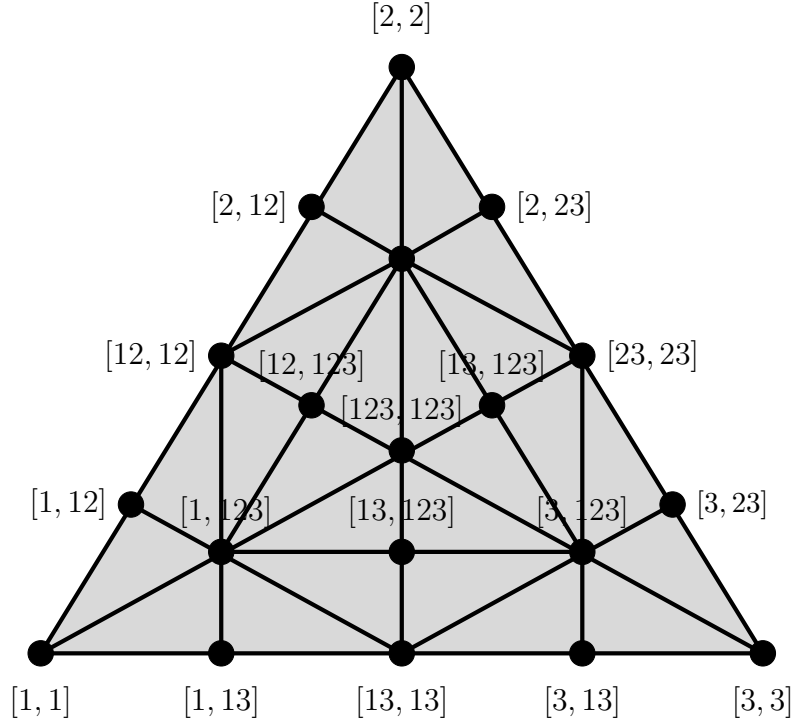

\begin{ej}Let $\Delta$ be a $2$-simplex on the ground set
$\{1,2,3\}$.  Then  $f(\Delta)=(3,3,1)$. The ground set of
$\Int(\Delta)$ is the set of all possible intervals in
$\Delta\setminus \{\emptyset\}$, see Figure 1. Therefore,
$f_0({\Int(\Delta)})=\sum_{l=0}^2(2^{l}-1)f_{l-1}(\Delta)=19, f_1({\Int(\Delta)})=
\sum_{l=0}^{2}\sum_{i=0}^{1} (-1)^i
{{1}\choose{i}}\big{[}(4-2i)^{l}-(3-2i)^{l}\big{]}f_{l-1}(\Delta)=42$ and
$f_2({\Int(\Delta)})= \sum_{l=0}^{2}\sum_{i=0}^{2} (-1)^i
{{2}\choose{i}}\big{[}(6-2i)^{l}-(5-2i)^{l}\big{]}f_{l-1}(\Delta)=24$.
\end{ej}

\section{The $h$-vector Transformation}

In this section, we represent the $h$-vector of  an interval subdivision  in term of the $h$-vector of the given simplicial complex.
Recall that the \textit{$h$-vector} $h(\Delta):=(h_0(\Delta),\ldots, h_d(\Delta))$ of $(d-1)$-simplicial complex $\Delta$  is defined in terms of the $f$-vector as
$$h_k(\Delta):=\sum_{i=0}^{k}(-1)^{k-i}{d-i \choose k-i}f_{i-1}(\Delta).$$ The \textit{$h$-polynomial} of $\Delta$ is defined as $$h(\Delta,x):=\sum_{i=0}^{d}h_i(\Delta)x^i.$$
We need to recall some notions to give the combinatorial description of the entries of  $h$-vector transformation matrix.

\subsection*{Signed Permutation Group $B_d$.} We  present here some definitions and notations  for the classical Weyl group of type $B$ (also known as the hyperoctahedral group or the signed
permutations group), denoted as $B_d$.  It is the group consisting of
all the bijections $\sigma$ of the set $ \{\pm 1,\ldots,\pm d\}$ onto itself such that $\sigma_{-i}= -\sigma_{i}$ for all $i\in \{\pm 1,\ldots, \pm d\}$.  The group $B_d$ can be viewed as a subgroup of $A_{2d}$ and the element $\sigma\in B_d$ is completely determined by $\sigma_1,\ldots, \sigma_d$. In one-line notation, we write $\sigma=\sigma_1\ldots \sigma_d$.
 For $\sigma\in B_d$, the \textit{descent set} is defined as  $$\textrm{Des}_B(\sigma):=\{i\in [0,d-1] : \sigma_i>\sigma_{i+1}\},$$ where $\sigma_0=0$ and the \textit{descent number} of type $B$ is defined  as $\textrm{des}_B(\sigma):=|\hbox{Des}_B(\sigma)|$.\\
Let  $$B_{d,j}:=\{\sigma\in B_d : \sigma_1=j\}$$ and define the $j$-Eulerian polynomial of type $B$
 as  \begin{equation}\label{B-Euler}
    B_{d,j}(t):=\sum_{\sigma\in B_{d,j}}t^{\des_B(\sigma)}=\sum_{k=0}^{d}B(d,j,k)t^k,
  \end{equation}
   where $B(d,j,k)$
 is the  number of elements in $B_{d,j}$ with exactly $k$ descents.\\

Let $$B_d^{+}:=\{\sigma\in B_d : \sigma_{d}>0\}$$ and
$$B_{d,j}^{+}:=\{\sigma\in B_d^+ : \sigma_1=j\}.$$
Similarly, for the other half hyperoctahedral group $B^-_d$, define
 $$B_d^{-}:=\{\sigma\in B_d : \sigma_{d}<0\}$$ and
$$B_{d,j}^{-}:=\{\sigma\in B_d^- : \sigma_1=j\}.$$
  Let us define the
  \textbf{\textit{$j$-Eulerian polynomials of type $B^+$}} by
  \begin{equation}\label{B+}
    B^+_{d,j}(t):=\sum_{\sigma\in B^+_{d,j}}t^{\des_B(\sigma)}=\sum_{k=0}^{d-1}B^+(d,j,k)t^k,
  \end{equation}
  where $B^+(d,j,k)$
 is the  number of elements in $B_{d,j}^{+}$ with exactly $k$ descents.  Similarly, define the
  \textbf{\textit{$j$-Eulerian polynomials of type $B^-$}} by
  \begin{equation}\label{B-}
    B^-_{d,j}(t):=\sum_{\sigma\in B^-_{d,j}}t^{\des_B(\sigma)}=\sum_{k=1}^{d}B^-(d,j,k)t^k,
  \end{equation}
  where $B^-(d,j,k)$
 is the number of elements in $B_{d,j}^{-}$ with exactly $k$ descents.\\
  Since $B_{d,j}=B^+_{d,j}\cup B^-_{d,j}$ so the $j$-Eulerian polynomial $B_{d,j}$ of type $B$ is $$B_{d,j}(t)=B^+_{d,j}(t)+B^-_{d,j}(t).$$
Here, we list $B^+_{d,j}(t)$  for $ d= 4$  and $1\leq s\leq 4$:\\
\begin{align*}
  B^+_{4,1}(t)= & 1+16t+7t^2  \\
  B^+_{4,2}(t)= & 14t+10t^2\\
B^+_{4,3}(t)= & 10t+14t^2\\
B^+_{4, 4}(t)= & 7t+16t^2+t^3
\end{align*}

\subsection*{The $h$-Vector of Interval Subdivision.} Let $\mathcal{H}_d$ be the transformation matrix from the $f$-vector to the $h$-vector, then \begin{equation*}\mathcal{H}_d=[h_{i,j}]_{0\leq i,j\leq d}=[(-1)^{i+j}{d-i\choose j-i}]_{0\leq i,j\leq d}\end{equation*} with its inverse transformation matrix
\begin{equation*}\mathcal{H}^{-1}_d=\big[{d-i\choose j-i}\big]_{0\leq i,j\leq d}\end{equation*}

Thus, $$h({\Int(\Delta)})=\mathcal{H}_d\mathcal{F}_d\mathcal{H}^{-1}_d h(\Delta)$$
Set it by $$\mathcal{R}_d=\mathcal{H}_d\mathcal{F}_d\mathcal{H}^{-1}_d.$$  For example,
$$\mathcal{R}_4=\left(
              \begin{array}{ccccc}
                1  & 0  &  0 & 0  & 0\\
                61 & 46 & 32 & 22 & 15\\
                115& 124& 128& 124 & 115\\
                15 & 22 & 32 & 46 & 61\\
                0  & 0  & 0  & 0  & 1\\
              \end{array}
            \right)
$$
In the following theorem, we show that this transformation matrix possesses a nice combinatorial description.
\begin{teo}\label{h}
The entries of  the matrix $\mathcal{R}_d$ are given as:
  $$\mathcal{R}_d=[B^+(d+1,s+1,r)]_{0\leq s,r\leq d}.$$
Thus, the $h$-vector of $\Int(\Delta)$ will be
$$h_r({\Int(\Delta)})=\sum_{s=0}^{d}B^+(d+1,s+1,r)h_s({\Delta}).$$
\end{teo}
To prove the above theorem, we need the following lemma regarding the recurrence relation of the entries $b_{r,l}$ of the matrix $\mathcal{F}_d$.
\begin{lema}\label{b}
For $1\leq r\leq d-1$ and $1\leq l\leq d$,
  $$
\sum_{i=1}^{l}2^i{l\choose i}b_{r,l-i}=b_{r+1,l}.$$
\end{lema}
\begin{proof}
Using \eqref{matrix}, we have
  \begin{align*} \nonumber
                                                        \nonumber\sum_{i=1}^{l}2^i{l\choose i}b_{r,l-i} = & \sum_{i=1}^{l}2^i{l\choose i}\sum_{j=0}^{r-1}(-1)^j{r-1 \choose j}[(2r-2j)^{l-i}-(2r-2j-1)^{l-i}] \\
                                                        = &\nonumber \sum_{j=0}^{r-1}(-1)^j{r-1 \choose j}\sum_{i=1}^{l}2^i{l\choose i}[(2r-2j)^{l-i}-(2r-2j-1)^{l-i}]\\
                                                         =&\nonumber   \sum_{j=0}^{r-1}(-1)^j{r-1 \choose j}[(2r-2j+2)^{l}-(2r-2j+1)^{l}-(2r-2j)^{l}+(2r-2j-1)^{l}]\\
\end{align*}
By re-summing, we get

\begin{align*}
                                                         = & (2r+2)^l-(2r+1)^l+\sum_{j=0}^{r-2}(-1)^{j+1}\big({r-1 \choose j}+ {r-1\choose
                                                         j+1}\big)[(2r-2j)^{l}\\
                                                         & -(2r-2j-1)^{l}]+(-1)^r(2^l-1)\\
                                                          = & \sum_{j=0}^{r}(-1)^j{r \choose j}\big[(2(r+1)-2j)^{l}-(2(r+1)-2j-1)^{l}\big]\\
                                                         =& b_{r+1,l}
\end{align*}
\end{proof}

\subsection*{Proof of Theorem \ref{h}}

Let $C_{r,s}$ be the $(r,s)$-entry of the matrix $C=\mathcal{F}_d\mathcal{H}^{-1}_d$. We have $$C_{r,s}=\sum_{l=0}^{d}{d-s\choose l-s}b_{r,l}=\sum_{l=0}^{d}{d-s\choose d-l}b_{r,l}$$

Let $\mathcal{C}_{r,s}$ denote the set of all set  partitions
$A=A_0|A_1|\ldots|A_r$ of rank $r$ of  $d+1$ elements  ranging from $\{\pm1,\pm
2,\ldots,\pm (d+1)\}$ for which exactly one of $\pm i$ appears in $A$ with
$\min A_0=s+1$ and $\max A_r:= \hbox{ the maximum element of } A_r>0$.  Fix $l$ and count partitions with  $d-l$ elements from  $\{s+2, \ldots, d+1\}$ in the first block $A_0$ along with $\min A_0 =s+1$. This can be done in ${{d-s} \choose {d-l}}$
ways. For $r=0$, we  have ${{d-s} \choose {d}}={{d-s} \choose
{-s}}=C_{0,s}$ such set partitions. For $r>0$, to form $A_1|\ldots |A_r$
we need to create a set partition from the remaining $l$ elements,
and this can be done in $b_{r,l}$ ways. Lets prove this claim by
induction on $r$. For $r=1$, to form $A_1$, we need to put $l$
elements from $\{\pm 1,\ldots,\pm (d+1)\}\setminus A_0$   such
that $\max A_1>0$. This gives $2^l-1$ choices, which is the same as
$b_{1,l}$. Suppose that the number of such set partitions
$A_1|\ldots|A_r$ of $l$ elements from $\{\pm 1, \ldots, \pm (d+1)\}$  is
$b_{r,l}$. Now, to form such set partition $A_1|\ldots|A_{r+1}$ of  $l$ elements, we first choose $i$ elements from $l$
remaining elements, where  $i>0$. This can be done in $2^i{l\choose
i}$ ways; and the set partition  $A_2|\ldots|A_{r+1}$ from remaining
$l-i$ elements can be done in $b_{r,l-i}$ ways (by induction
hypothesis). Thus we have $\sum_{i=1}^{l}2^i{l\choose i}b_{r,l-i}$
ways to form the required  set partitions of rank $r+1$ of $l$
elements.  By Lemma \ref{b}, we have $$
\sum_{i=1}^{l}2^i{l\choose i}b_{r,l-i}=b_{r+1,l}.$$

Thus, $|\mathcal{C}_{r,s}|=C_{r,s}$, the $(r,s)$-entry of the matrix $C$.

Let $$\mathcal{C}_{s}=\cup_{r=0}^{d}\mathcal{C}_{r,s}$$ be the collection of all set partitions of $d+1$ elements
 from $\{\pm1,\pm 2,\ldots,\pm (d+1)\}$ for which exactly one of $\pm i$ appears in $A$ with
  $\min A_0=s+1$ and $\max A_r>0$. Let $C_s(t)$ denote the generating function counting these set
 partitions according to the number of bars,$$C_s(t) = \sum_{A\in \mathcal{C}_{s}}t^{\rank A}=\sum_{r=0}^{d}C_{r,s}t^r. $$
 It can be noted that $C_s(t)$ is also  the generating function for $s$-th column of the matrix $C$. Also, a polynomial $M_s(t)$ can be viewed as multiplication of matrices:
 $$M_s(t)=M_s\cdot \mathbf{t},$$ where $M_s$ is the $s$-th column of the matrix $M$ and $\mathbf{t}$ denote the column vector of powers of $t$, $\mathbf{t}=(1,t,\ldots,t^d)^T$.\\

It can be noted that each set partition $A = A_0|A_1| \ldots |A_r$ can be mapped to a permutation
$\sigma = \sigma(A)$ by removing bars and writing each block in increasing
order. Since $\min A_0 = s +1$, this means $\sigma_{1} = s+1$, and $\max A_r>0$ means $\sigma_{d+1}>0$. That is, $\sigma\in B_{d+1,s+1}^{+} $.
Further, $\textrm{Des}_B(\sigma)\subset D$, where $D = D(A) = \{|A_0|, |A_0|+|A_1|, \ldots, |A_0|+|A_1|+\ldots +|A_{r-1}|\}$, i.e., there must be bars in $A$ where there are descents in $\sigma$. So, we can write

\begin{align*}
C_s(t)& = \sum_{A\in \mathcal{C}_{s}}t^{\rank A}\\
&=\sum_{I\subset \{1,\ldots, d\}}\sum_{A\in \mathcal{C}_{s},D(A)=I}t^{|I|}\\
&=\sum_{I\subset \{1,\ldots, d\}}\sum_{\sigma\in B^+_{d+1,s+1},\textrm{Des}_B(\sigma)\subset I}t^{|I|}\\
&=\sum_{\sigma\in B^+_{d+1,s+1}}\sum_{\textrm{Des}_B(\sigma)\subset I}t^{|I|}\\
&=\sum_{\sigma\in B^+_{d+1,s+1}}t^{\mathrm{des}_B(\sigma)}(1+t)^{d-\mathrm{des}_B(\sigma)}\\
&=(1+t)^d B^+_{d+1,s+1}(t/(1+t)).
\end{align*}
Since $C_s(t)$ reads the $s$-th column of $\mathcal{F}_d\mathcal{H}^{-1}_d$ , the polynomial $ \mathcal{H}_dC_s(t)=\mathcal{H}_dC_s\cdot \mathbf{t}=
B^+_{d+1,s+1}(t)$ reads the $s$-th column  of $\mathcal{R}_d=\mathcal{H}_d\mathcal{F}_d\mathcal{H}^{-1}_d$. Thus the columns of $\mathcal{R}_d$
are encoded by the $j$-Eulerian polynomials of type $B^+$.\, \, \, \, \, \, \, \, \, \, \, \, \, \, \, \, \, \, \, \, \, \, \, \, \, \, \, \, \,  $\square$\\

\begin{coro}
Let $\Delta$ be a $(d-1)$-dimensional simplicial complex such that $h_r(\Delta)\geq 0$ for all $0\leq r\leq d$. Then for $0\leq r\leq d$, $$h_r{(\Int(\Delta))}\geq h_r(\Delta).$$
\end{coro}
\begin{proof}
  By  Theorem \ref{h}, we have $$h_r({\Int(\Delta)})=\sum_{s=0}^{d}B^+(d+1,s+1,r)h_s({\Delta}).$$ Since $B^+(d+1,r+1,r)\geq 1$ and $h_r(\Delta)\geq 0$  for all $r$,   therefore the result follows.
\end{proof}
\begin{ej}
  The above corollary  is not true in general. Consider a $3$-dimensional simplicial complex $\Delta=\langle 1234, 125, 345\rangle$ with $f$-vector $f(\Delta)=(1,5,10,6,1)$ and $h$-vector $h(\Delta)=(1,1,1,-3,1)$.
 Then,  $f({\Int}(\Delta))=(1,92,380,480,192)$ and $h({\Int}(\Delta))=(1,88,110,-8,1)$.
\end{ej}

Some basic facts about the numbers $B^+(d,s,r)$ are given in the following lemma:
\begin{lema}\label{B}
The following relations hold for $d\geq 1$, $1\leq s\leq d,\  0\leq r\leq d-1$:
\begin{enumerate}

    \item $$\sum_{r=0}^{d-1}B^+(d,s,r)=2^{d-2}(d-1)!.$$
    \item $$B^+(d,s,r)=B^+(d,d-s+1,d-r-1).$$
    \item $$B^+(d,s,r)=\sum_{j=1}^{d-1}B^+(d-1,-j,r)+\sum_{j=s}^{d-1}B^+(d-1,j,r)+\sum_{j=1}^{s-1}B^+(d-1,j,r-1).$$ Thus, the recurrence  relation holds:
    $$B^+_{d,s}(t)=t\sum_{j=1}^{s-1}B^+_{d-1,j}(t)+\sum_{j=s}^{d-1}B^+_{d-1,j}(t)+\sum_{j=1}^{d-1}B^+_{d-1,-j}(t),$$ with initial conditions $B^+_{1,1}(t)=1$ and $B^+_{1,-1}(t)=0$.
    \item  $$B^+(d,-s,r)=\sum_{j=1}^{d-1}B^+(d-1,j,r-1)+\sum_{j=s}^{d-1}B^+(d-1,-j,r-1)+\sum_{j=1}^{s-1}B^+(d-1,-j,r).$$ Thus, the recurrence  relation holds:
    $$B^+_{d,-s}(t)=t[\sum_{j=1}^{d-1}B^+_{d-1,j}(t)+\sum_{j=s}^{d-1}B^+_{d-1,-j}(t)]+\sum_{j=1}^{s-1}B^+_{d-1,-j}(t).$$

  \end{enumerate}
\end{lema}
\begin{proof}
(1) This follows from the definition of $B^+(d,s,r)$.
\\ (2) Let $\mathcal{B}^+(d,s,r)$  denote the set of all elements $\sigma\in B^+_{d,s}$ such that $\des_B(\sigma)=r$. There is a bijection between the sets
$\mathcal{B}^+(d,s,r)$ and $\mathcal{B}^+(d,d-s+1,d-r-1)$ given by
$\sigma=(\si_1,\ldots,\si_d) \mapsto
\bar{\sigma}=(\bar{\si}_1,\dots,\bar{\si}_d)$, where
$$\bar{\si}_i:=\left\{
               \begin{array}{ll}
                 d+1-\si_i, & \hbox{$\si_i>0$;} \\
                 -( d+1+\si_i), & \hbox{$\si_i<0$.}
               \end{array}
             \right .$$

Consider the following  three possible cases:
\begin{itemize}
  \item $\si_i>0>\si_{i+1}$ iff $\bar{\si}_i>0>\bar{\si}_{i+1}$
  \item $\si_i>\si_{i+1}>0$  iff  $\bar{\si}_{i+1}>\bar{\si}_{i}>0$
  \item $0>\si_i>\si_{i+1}$  iff  $0>\bar{\si}_{i+1}>\bar{\si}_{i}$
\end{itemize}
In the first case, $i\in \Des_B(\si) $ iff $i\in \Des_B(\bar{\si}) $ and in other two cases, we have $i\in \Des_B(\si) $ iff $i\notin \Des_B(\bar{\si}) $. It is clear that $0$ is not descent of $\si$ and $\bar{\si}$.  Thus, $\des_B(\sigma)+\des_B(\bar{\sigma})=d-1$.\\
  (3)  The recursion formula  follows from the effect of removing $\sigma_1=s$ from the signed permutation $\sigma$ in $B^+_{d,s}$ with $\des_B(\sigma)=r$.\\
  The proof of (4) is similar to (3).
\end{proof}

Recall  that the $h$-vector $h(\Delta)$  of a $(d-1)$-dimensional simplicial  complex $\Delta$ is reciprocal
if $h_i(\Delta)=h_{d-i}(\Delta)$ for $0\leq i\leq d$. This  condition is equivalent to the $h$-vector satisfying the Dehn-Sommerville
relations.

\begin{coro}\label{reciprocal}
  Let $\Delta$ be a $(d-1)$-dimensional simplicial complex with a reciprocal $h$-vector. Then $\Int{(\Delta)}$  has also a reciprocal $h$-vector.
\end{coro}
\begin{proof}
The result  follows from above Lemma \ref{B}(2) and Theorem \ref{h}.
\end{proof}
\subsection*{Some Properties of the Transformation Matrices }
In this subsection, we describe some properties of the transformation matrices. The motivation to study the eigenvalues and eigenvectors of these matrices is to study the limit behaviour of the roots under successive interval subdivisions, which is discussed in the next section. We know from  Theorem \ref{f} and Theorem \ref{h} that
$$f({\Int(\Delta)})=\mathcal{F}_d f({\Delta})$$
and
$$h({\Int(\Delta)})=\mathcal{R}_d h({\Delta}).$$
\begin{lema}\label{diagonalization}
Let $d\geq 1$.
\begin{enumerate}
  \item The matrices $\F_d$ and $\R _d$ are similar.
  \item   The matrices $\F_d$ and $\R _d$ are diagonalizable with eigenvalues $1$ of multiplicity $2$ and eigenvalues $2\cdot 2!, 2^2\cdot 3!, \ldots,  2^{d-1}\cdot d!$ of  multiplicity $1$.
\end{enumerate}
\end{lema}
\begin{proof}
First assertion follows from the facts that the transformation from $f({\Delta})$ to $h({\Delta})$ is an invertible linear transformation and by Theorem \ref{f} and Theorem \ref{h}. The second assertion follows as  $\F_d$ is an upper triangular matrix with diagonal $1, 1, 2\cdot 2!, 2^2\cdot 3!, \ldots, 2^{d-1}\cdot d!$; the first and the second unit vectors are eigenvectors for the eigenvalue $1$.
\end{proof}
The next result follows from the fact that $\F_{d+1}$ and $\F_d$ are upper triangular and that if
one deletes the $(d + 2)$-nd column and row from $\F_{d+1}$ then one obtains $\F_d$.
\begin{lema}
Let $d \geq 1$ and  $v_1^1,v_1^2,v_{2\cdot 2!},\ldots , v_{2^{d-1}\cdot d!}$
 be a basis of $\mathbb{R}^{d+1}$ consisting of eigenvectors of the matrix $\F_d$, where $v_1^1, v_1^2$  are eigenvectors for the eigenvalue $1$
and $v_{2^i\cdot (i+1)!}$ is an eigenvector for the eigenvalue $2^i\cdot(i+1)!$, $1 \leq i \leq d-1$. Then the vectors
$(v_1^1,0), (v_1^2,0), (v_{2\cdot 2!},0),\ldots , (v_{2^{d-1}\cdot d!},0)$ are eigenvectors of $\F_{d+1}$ for the eigenvalues $1, 1, 2\cdot 2!, \ldots , 2^{d-1}\cdot d!$.
\end{lema}
The following result follows due to    \cite[Lemma 6]{brenti2008f} and above lemmas.
\begin{lema}\label{eigenvector}
Let $\Delta$ be a $(d - 1)$-dimensional simplicial complex. Let $w_1^1, w_1^2, w_{2\cdot 2!}, \ldots,  w_{2^{d-1}\cdot d!}$
 be a basis of eigenvectors of the matrix $\R_d$, where $w_1^1, w_1^2$  are eigenvectors for the eigenvalue $1$
and $w_{2^i.(i+1)!}$ is an eigenvector for the eigenvalue $2^i\cdot (i+1)!$, $1 \leq i \leq d-1$.
\begin{enumerate}
\item  If we expand $h({\Delta}) = a^1_1w_1^1+a^2_1w_1^2+\sum_{i=1}^{d-1}a_{2^i\cdot (i+1)!}w_{2^i\cdot (i+1)!}$ in terms of the eigenvectors, then $a_{2^{d-1}\cdot d! }    \neq 0.$
  \item The first and  last entry in $w_{2\cdot 2!}, \ldots,  w_{2^{d-1}\cdot d!}$ are zero.
  \item The vectors $w_1^1$ and $w_1^2$ can be chosen such that
$w_1^1= (1, i_1, \ldots , i_{2^{d-2}\cdot(d-1)!}, 0)$ and $w_1^2= (0, j_1, \ldots , j_{2^{d-2}\cdot(d-1)!}, 1)$, for real numbers $i_l,j_k$ for all $l,k$.
  \item The vector $w_{2^{d-1}\cdot d!}$ can be chosen such that $w_{2^{d-1}\cdot d!} = (0, b_1, \ldots, b_{2^{d-2}\cdot(d-1)!}, 0)$ for strictly positive
rational numbers $ b_{2^{i}\cdot(i+1)!}$, $0 \leq i \leq d-2$.
\end{enumerate}
\end{lema}
\begin{remark}
  It can  easily be seen that the eigenvector $(v_1, \ldots,  v_{d+1})$ of $\F_d$ for the eigenvalue different from $2^{d-1}\cdot d!$ satisfies the identity $\sum_{i=1}^{d+1}v_i=0$.
\end{remark}
\subsection*{Limit behavior of the roots}
Let $\De$ be a $(d-1)$-dimensional simplicial complex. For $n\geq 1$, let $\Int^{(n)}(\De)$ denote the $n$-th interval subdivision of $\De$.
\begin{teo}\label{limit}
  The roots of $f$-polynomials of successive interval subdivisions of $\De$ converge to fixed values depending only on dimension of $\De$.
\end{teo}
\begin{proof}
The assertion follows from Lemma \ref{diagonalization} and from arguments used in the proof of Theorem 3.7( for barycentric case) in \cite{delucchi2012face}.

\end{proof}

\section{Real-rootedness of $h$-polynomial}

We start this section with the description of the $j$-Eulerian
polynomial of type $B^+$ and $B^-$.  Lets recall that a polynomial
$p(x)=\sum_{i=0}^{d}a_ix^i$ with real coefficients is
\textit{unimodal}
 if there exists $ j$ such that $a_0\leq a_1\cdots \leq  a_{j}\geq a_{j+1}\geq \cdots \geq a_d.$ A polynomial $p(x)$
 is said to be \textit{log-concave} if  $a_i^2\geq a_{i-1}a_{i+1}$ for $1\leq i\leq d-1$. It is well-known that if $p(x)$ is
  real-rooted with non-negative coefficients, then $p(x)$ is log-concave and
  unimodal. Here, we present the main result of this section.

\begin{teo}\label{h-poly}

  Let $\Delta$ be a $(d-1)$-dimensional simplicial complex such that $h$-vector $h(\Delta)=(h_0(\Delta),\ldots, h_d(\Delta))$ is non-negative.
  Then the $h$-polynomial  $$h(\Int(\Delta),t)=\sum_{i=0}^{d}h_i(\Int(\Delta))t^i$$ has only real roots.
  In particular, the $h$-polynomial $h(\Int(\Delta),t)$ is  log-concave, and hence unimodal.
\end{teo}
Towards the proof, we give some definitions and results
about  descent  statistics on the symmetric group $A_d$ and the
hyperoctahedral group $B_d$. Recall that for $\si\in A_d$, the
descent set $\Des_A(\si)=\{i\in [d-1]\ :\ \si_i>\si_{i+1}\}$ and the
descent number $\des_A(\si)=|\Des_A(\si)|$. Set $A_{d,j}=\{\sigma \in A_d :\ \sigma_1=j\}$. Denote the
$j$-Eulerian polynomial of type $A$ as
\begin{equation}\label{A}
  A_{d,j}(t)=\sum_{k=0}^{d-1}A(d,j,k)t^k,
\end{equation}
where $A(d,j,k)$ is the number of permutations $\si\in A_d$ with
$\si_1=j$ and $\des_A(\si)=k$. The next result gives the recurrence
relation for $A_{d,j}(t)$. For $j=1$, it is already known due to
\cite{10.2307/2319133} and \cite{garsia1979maj}.
\begin{lema}
   For $d\geq 1$ and $1\leq j< d$, we have
  \begin{equation}\label{Arecurrence}
 A_{d,j}(t)=(1+t(d-2))A_{d-1,j}+t(1-t)\frac{d}{dt}(A_{d-1,j}(t))
  \end{equation}
  \end{lema}
  \begin{proof}
Let $\si=j\si_2 \cdots \si _{d-1}\in A_{d-1,j}$. For $i \in \{2,
\ldots, d\}$, define   $\si^i$  in $A_{d,j}$  obtained from $\si$ by
inserting $d$
 at the $i$th position. For $2 \leq i \leq d $, we have  $\si^i\in A_{d,j}$  if and only if $\si\in A_{d-1,j}$.  Moreover, for  $2 \leq i \leq d-1$, we have
$$\des_A(\si^{ i})=\left\{
                          \begin{array}{ll}
                            \des_A(\si), & \hbox{if $i-1\in \Des_A(\si)$;} \\
                            \des_A(\si)+1 , & \hbox{otherwise.}
                          \end{array}
                        \right.
 $$
and $ \des_A(\si)= \des_A(\si^{d})$. Therefore,
\begin{align*}
  A_{d,j}(t)= & \sum_{\si\in A_{d,j}}t^{ \des_A(\si)}=\sum_{i=2}^{d-1}\big(\sum_{\si\in A_{d-1,j}}t^{ \des_A(\si^i)}\big)+ \sum_{\si\in A_{d-1,j}}t^{ \des_A(\si^{d})}\\
  = &  \sum_{\si\in A_{d-1,j}}\big(\des_A(\si)t^{\des_A(\si)}+(d-2-\des_A(\si))t^{\des_A(\si)+1}\big)+A_{d-1,j}(t)\\
  = & (d-2)\sum_{\si\in A_{d-1,j}}t^{\des_A(\si)+1}+(1-t)\sum_{\si\in A_{d-1,j}} \des_A(\si)t^{\des_A(\si)}+A_{d-1,j}(t)\\
  = & (1+(d-2)t)A_{d-1,j}(t)+t(1-t)\frac{d}{dt}(A_{d-1,j}(t)).
\end{align*}
 \end{proof}
The next result gives the recurrence relations for $B^+_{d,j}(t)$ and
$B^-_{d,j}(t)$. It has already been proved in
\cite{athanasiadis2013symmetric} for $j=1$.
\begin{lema}
  For $d\geq 1$ and $1\leq j\leq d$, we have
  \begin{equation}\label{B+recurrence}
  B^+_{d,j}(t)=2(d-2)tB^+_{d-1,j}(t)+2t(1-t)\frac{d}{dt}(B^+_{d-1,j}(t))+B_{d-1,j}(t)
  \end{equation} and
\begin{equation}\label{B-recurrence}
  B^-_{d,j}(t)=2(d-2)tB^-_{d-1,j}(t)+2t(1-t)\frac{d}{dt}(B^-_{d-1,j}(t))+tB_{d-1,j}(t)
  \end{equation}
\end{lema}
\begin{proof}
  We will prove the relation \eqref{B-recurrence} and the proof of \eqref{B+recurrence} follows from the symmetry. Let $\si=j\si_2\cdots \si _{d-1}\in B_{d-1,j}$.
  For $i \in \{2, \ldots, d\}$, define   $\si^i$ and $\si^{-i}$ in $B_{d,j}$  obtained from $\si$ by inserting $d$ and  $-d$ respectively at the $i$th position. For $2 \leq i \leq d-1$, we have  $\si^i\in B^-_{d,j}$ (respectively, $\si^{-i}\in B^-_{d,j}$) if and only if $\si\in B^-_{d,j}$. On the other hand,  $\si^{-d}\in B^-_{d,j}$ and $\si^d\in B^+_{d,j}$  for every $\si\in B_{d,j}$. Moreover, for  $2 \leq i \leq d-1$, we have
 $$\des_B(\si^{\pm i})=\left\{
                         \begin{array}{ll}
                            \des_B(\si), & \hbox{if $i-1\in \Des_B(\si)$;} \\
                           \des_B(\si)+1 , & \hbox{otherwise.}
                          \end{array}
                        \right.
  $$ and $ \des_B(\si)= \des_B(\si^{-d})-1$. Therefore,
\begin{align*}
  B^-_{d,j}(t)= & \sum_{\si\in B^-_{d,j}}t^{ \des_B(\si)}=\sum_{i=2}^{d-1}\big(\sum_{\si\in B^-_{d-1,j}}t^{ \des_B(\si^i)}+t^{ \des_B(\si^{-i})}\big)+\sum_{\si\in B_{d-1,j}}t^{ \des_B(\si^{-d})} \\
  = & 2 \sum_{\si\in B^-_{d-1,j}}\big(\des_B(\si)t^{\des_B(\si)}+(d-2-\des_B(\si))t^{\des_B(\si)+1}\big)+tB_{d-1,j}(t)\\
  = & 2(d-2)\sum_{\si\in B^-_{d-1,j}}t^{\des_B(\si)+1}+2(1-t)\sum_{\si\in B^-_{d-1,j}} \des_B(\si)t^{\des_B(\si)}+tB_{d-1,j}(t)\\
  = & 2(d-2)tB^-_{d-1,j}(t)+2t(1-t)\frac{d}{dt}(B^-_{d-1,j}(t))+tB_{d-1,j}(t).
\end{align*}
  \end{proof}
\begin{lema}
  For $d\geq 1$ and $1\leq j\leq d$, we have
\begin{equation}\label{B+B-}
  B^+_{d,j}(t)=t^dB^-_{d,-j}(t^{-1})
\end{equation}and
 \begin{equation}\label{B-B+}
  B^+_{d,-j}(t)=t^dB^-_{d,j}(t^{-1})
\end{equation}
\end{lema}
\begin{proof}
  There is a bijection between $B^+_{d,j}$ and $B^-_{d,-j}$ given as $\si\mapsto \bar{\si}$, where\linebreak $\bar{\si}=(-j,-\si_2,\cdots,-\si_d)$.
   It is clear that $\si\in B^+_{d,j}$ if and only if  $\bar{\si}\in B^-_{d,-j}$. It is straightforward to show that  $\des_B(\si)+\des_B(\bar{\si})=d$, therefore, \eqref{B+B-} follows. Similarly, \eqref{B-B+} also holds.
\end{proof}
Let $$T_{d,j}(t):=B^+_{d,j}(t^2)+\frac{1}{t}B^-_{d,j}(t^2).$$
 Since  for any $\si\in B^-_{d,j}$, $\des_B(\si)$ is always strictly positive, there is no constant term involved in $B^-_{d,j}(t^2)$.
 So, the right hand side of the above expression is a polynomial.
 \begin{teo}
For $d\geq 1$ and $1\leq j\leq d$, we have
\begin{equation}\label{Trecurrence}
T_{d,j}(t)=[1+t+(d-2)t^2]T_{d-1,j}(t)+t(1-t^2)\frac{d}{dt}(T_{d-1,j}(t))
  \end{equation}
  \end{teo}
\begin{proof}
Using product and chain rule, we have the following
\begin{equation}\label{dT}
  \frac{d}{dt}(T_{d-1,j}(t))=2t[\frac{d}{dt}(B^+_{d-1,j}(t^2))+ \frac{1}{t}\frac{d}{dt}(B^+_{d-1,j}(t^2))]-\frac{1}{t^2}B^+_{d-1,j}(t^2)
\end{equation}
  Using \eqref{B+recurrence}, \eqref{B-recurrence} and \eqref{dT}, we have
  \begin{align*}
     T_{d,j}(t)=&B^+_{d,j}(t^2)+\frac{1}{t}B^-_{d,j}(t^2) \\
     =&  2(d-2)t^2B^+_{d-1,j}(t^2)+2t^2(1-t^2)\frac{d}{dt}(B^+_{d-1,j}(t^2))+B_{d-1,j}(t^2)+ \\
     & \frac{1}{t}[2(d-2)t^2B^-_{d-1,j}(t^2)+2t^2(1-t^2)\frac{d}{dt}(B^-_{d-1,j}(t^2))+t^2B_{d-1,j}(t^2)] \\
    = & 2(d-2)t^2T_{d-1,j}(t)+2t^2(1-t^2)\frac{1}{2t}[\frac{d}{dt}(T_{d-1,j}(t))\\
     & +\frac{1}{t^2}B^-_{d-1,j}(t^2)]+(1+t)B_{d-1,j}(t^2)\\
    = &=2(d-2)t^2T_{d-1,j}(t)+t(1-t^2)\frac{d}{dt}(T_{d-1,j}(t))+\frac{1}{t}(1-t^2)B^-_{d-1,j}(t^2)\\
     & +(1+t)[B^+_{d-1,j}(t^2)+B^-_{d-1,j}(t^2)]\\
     =& [1+t+(d-2)t^2]T_{d-1,j}(t)+t(1-t^2)\frac{d}{dt}(T_{d-1,j}(t)).
  \end{align*}
\end{proof}
The following is the key result for proving real rootedness of $B^+_j(t)$ and $B^-_j(t)$. It generalizes
the relation  \cite[Equation (11)]{yang2015real} for
$j=1$.
\begin{prop}
  For $d\geq 1$ and $1\leq j\leq d$,
  \begin{equation}\label{AB}
    (1+t)^{d-1}A_{d,j}(t)=T_{d,j}(t)
  \end{equation}
\end{prop}
\begin{proof}
  We will show that both sides of \eqref{AB} satisfy the same recurrence relation. It can be easily verified that the equality hold for $d\leq 2$.
 Let $$S_{d,j}(t):=(1+t)^{d-1}A_{d,j}(t).$$
 It is clear that
\begin{equation}\label{dS}
  \frac{d}{dt}(S_{d-1,j}(t))=(1+t)^{d-1}\frac{d}{dt}(A_{d-1,j}(t))+(d-1)S_{d-1,j}(t)
\end{equation}
 Using \eqref{Arecurrence} and \eqref{dS}, we have
  \begin{align*}
     S_{d,j}(t)=&(1+t)^{d-1} A_{d,j}(t) \\
   =  &  (1+t)^{d-1}[(1+t(d-2))A_{d-1,j}(t)+t(1-t)\frac{d}{dt}(A_{d-1,j}(t))] \\
   =  & (1+t)(1+t(d-2))S_{d-1,j}(t)+t(1-t^2)[\frac{d}{dt}(S_{d-1,j}(t))-(d-1)S_{d-1,j}(t)] \\
   =  &[1+t+(d-2)t^2]S_{d-1,j}(t)+t(1-t^2)\frac{d}{dt}(S_{d-1,j}(t)).
  \end{align*}
  which gives the same recurrence relation \eqref{Trecurrence} as $T_{d,j}$.
\end{proof}
Since $B^+_{d,j}(t^2)$ involves only even powers in $t$ and $\frac{1}{t}B^-_{d,j}(t^2)$ involves only odd powers in $t$, so we have the following corollary.
\begin{coro}\label{Cor-E2}
  For $d\geq 1$ and $1\leq j\leq d$, we have
 \begin{equation}\label{E2}
   B^+_{d,j}(t)=E_2\big((1+t)^{d-1}A_{d,j}(t)\big),
 \end{equation} and
\begin{equation}\label{E2+}
   B^-_{d,j}(t)=E_2\big(t(1+t)^{d-1}A_{d,j}(t)\big),
 \end{equation}
 where $E_r$ is the operator  on formal series  defined by
  \begin{equation}\label{Er}
E_r\big(\sum_{k\geq 0}c_kt^k\big)=\sum_{k\geq0}c_{rk}t^k.
\end{equation}
\end{coro}
Recall a result from \cite{brenti1989unimodal} which is a key tool to prove the real rootedness of the  polynomials $B^+_{d,j}(t)$ and $B^-_{d,j}(t)$.
\begin{teo}\label{Brenti} \cite[Theorem 3.5.4]{brenti1989unimodal}
  Let $p(x)=\sum_{i=0}^{m}a_ix^i$ be a polynomial having only real non-positive
zeros. Then for each $r\in \mathbb{N}$, the polynomial $E_r(p(x))$(defined in \eqref{Er}) has only real non-positive zeros.
\end{teo}

\begin{teo}
 The polynomials $B^+_{d,j}(t)$ and $B^-_{d,j}(t)$ are real-rooted for all $d\geq 1$ and $1\leq j\leq d$.
 \end{teo}
 \begin{proof}
 It follows from Corollary  \ref{Cor-E2}, Theorem \ref{Brenti}  and the fact
  that $A_{d,j}(t)$ are real-rooted, see \cite[Theorem 2]{brenti2008f}.
 \end{proof}


A collection of polynomials $f_1,f_2,\ldots, f_k\in \mathbb{R}[t]$
is said to be \textit{compatible} if for all non-negative real
numbers $c_1,c_2,\ldots,c_k$, the polynomial $\sum_{i=1}^{k}c_if_i$
has only real zeros. The polynomials $f_1,f_2,\ldots, f_k\in
\mathbb{R}[t]$ are pairwise compatible if for all $i,j\in
\{1,\ldots, k\}$, $f_i$ and $f_j$ are compatible. By
\cite[2.2]{chudnovsky2007roots}, the polynomials $f_1,f_2,\ldots,
f_k$ with positive leading coefficients are pairwise compatible iff
they are compatible. In \cite{visontai2013eulerian}, the
authors gave some conditions under which a set of compatible
polynomials are mapped to another set of compatible polynomials. More precisely:

\begin{teo}\label{Comp}\cite[Theorem 6.3]{visontai2013eulerian}
  Given a set of polynomials $f_1,f_2,\ldots, f_k\in \mathbb{R}[t]$ with positive leading
coefficients satisfying for all $1\leq i<j\leq k$ that

\begin{enumerate}
  \item $f_i$ and $f_j$ are compatible, and
  \item $tf_i$ and $f_j$  are compatible.
\end{enumerate}
 Define another set of polynomials $g_1,\cdots,g_{k'}\in \mathbb{R}[t]$ by
  $$g_l(t)=\sum_{i=0}^{n_l-1}tf_i+\sum_{i=n_l}^{k}f_i,$$
for $1\leq l\leq k'$, $0\leq n_1 \leq n_2\leq \cdots  n_{k'}\leq k$.
Then  for all $1\leq i<j\leq k'$, we have

\begin{description}
  \item[a] $g_i$ and $g_j$ are compatible, and
  \item[b] $tg_i$ and $g_j$  are compatible.
\end{description}

\end{teo}
\subsection*{Proof of Theorem \ref{h-poly}.} Let us fix the order of polynomials $B^+_{d,j}(t)$ for $j\in \{\pm1,\pm2,\cdots, \pm d\}$ to apply Theorem \ref{Comp}. Define $$f_i:=\left\{
                          \begin{array}{ll}
                            B^+_{d,i}(t), & \hbox{$1\leq i\leq d$;} \\
                            B^+_{d,i-2d-1}(t), & \hbox{$d+1\leq i\leq 2d$.}
                          \end{array}
                        \right.
$$

We claim that the set of polynomials  $\{f_i :\ 1\leq i\leq 2d\}$
is compatible. We show it  by induction on $d$. For $d=1$, it is trivial. For $d=2$, we have $f_1=1$ and $f_i=t$ for  $2\leq i\leq 4$.
It is clear that these polynomials are pairwise compatible. Moreover, $tf_i$ and $f_j$ for $1\leq i<j\leq 4$ are also compatible.

 By Lemma \ref{B} (3) and (4), the polynomials
   $f_i$ satisfy the recurrence relation which has the same form required in  Theorem \ref{Comp}.
   Therefore, by induction hypothesis, our claim is true. In particular,   $\{f_j=B^+_{d,j}(t)\ :\ 1 \leq j\leq d\}$ is
   compatible for all $d\geq 1.$ Since $h_i(\Delta)$ is non-negative for all $i$, the $h$-polynomial $$h(\Int(\Delta),t)=\sum_{i=0}^{d}h_i(\Delta)B^+_{d+1,i+1}(t)$$ is
   real-rooted.\, \, \, \, \, \, \, \, \, \, \, \, \, \, \, \, \, \, \, \, \, \, \, \, \, \, \, \, \, \, \, \, \, \, \, \, \, \, \, \, \, \, \, \, \, \, \, \, \, \,  \, \, \, \, \, $\square$\\

\noindent At this point, we are in position to relate our results to
the Charney-Davis Conjecture. A $(d-1)$-dimensional simplicial
complex $\Delta$ with a non-negative reciprocal $h$-vector satisfies
the Charney-Davis Conjecture if
$(-1)^{\lfloor\frac{d}{2}\rfloor}h(\Delta,-1) \geq 0$ holds.
\begin{coro}\label{Chraney}
  The Charney-Davis conjecture holds for the interval subdivision of a $(d-1)$-dimensional simplicial complex $\Delta$ for which $h_i(\Delta)\geq 0$ and $h_i(\Delta)=h_{d-i}(\Delta)$ for $0\leq i\leq d.$
\end{coro}
\begin{proof} Since $\Delta$ has $h_i(\Delta)\geq 0$ and $h_i(\Delta)=h_{d-i}(\Delta)$ for $0\leq i\leq d$ so by  Corollary \ref{reciprocal}, $\Int(\Delta)$  has a reciprocal $h$-polynomial $h(\Int(\Delta),t)$. By Theorem \ref{h-poly}, we
also know that $h(\Int(\Delta),t)$ has only real zeros. Since the coefficients of $h(\Int(\Delta),t)$ are non-negative
and $h_0(\Int(\Delta))=1$, it follows that the zeros of $h(\Int(\Delta),t)$ are all strictly negative. Therefore,  if $\beta$ is a zero of $h(\Int(\Delta),t)$ then $1/\beta$ is also a  zero.
Thus,  the zeros are either $-1$ or come in pairs $\beta < -1 < 1/\beta<0$.  If $-1$ is a zero of $h(\Int(\Delta),t)$ then the
assertion follows trivially. If $-1$ is not a zero then $d$ must be even and $$h(\Int(\Delta),-1)=\prod_{i=1}^{d/2}(-1-\beta_i)(-1-1/\beta_i),$$ where for all $1\leq i\leq d/2$, $\beta_i < -1 < 1/\beta_i<0$, 
which shows that
 $h(\Int(\Delta),-1)$ has sign $(-1)^{d/2}$. Thus $$(-1)^{d/2}h(\Int(\Delta),-1)\geq 0,$$
which implies the assertion.
\end{proof}
The following Remark \ref{successive} follows from Lemma \ref{diagonalization}, Lemma \ref{eigenvector}, Theorem \ref{h-poly} and  \cite[Theorem 3]{brenti2008f}.
\begin{remark}\label{successive}
  The roots of the $h$-polynomial of successive interval subdivisions converge to some $d-1$ non-negative real numbers which depend only on the dimension of $\De$.
\end{remark}

\subsection*{ Acknowledgment} We are grateful to Prof. Volkmar Welker for introducing the subject,  continuous support and guidance. We are deeply indebted to anonymous referees for their comments and suggestions based on careful observations on the earlier version, which improved this paper in a great deal. This project was started during the First Research School on Commutative Algebra and Algebraic Geometry (August $5-17, 2017$) in IASBS, Iran. We would like to thank the organizers for providing us the opportunity and warm hospitality.

\bibliographystyle{amsalpha}
\bibliography{References}
\end{document}